%% file: Main.tex
\documentclass[12pt]{amsart}
\input{common/packages.tex}
\input{common/notation.tex}
\input{common/environments.tex}
\input{common/collab.tex}

\begin{document}

\input{common/contacts}

\title{Affine cones over cubic surfaces \\are flexible in codimension one}

\thanks{The research was supported by the grant RSF-19-11-00172.}

\begin{abstract}
Let $Y$ be a smooth del Pezzo surface of degree 3 polarized by a very ample divisor that is not proportional to the anticanonical one. Then the affine cone over $Y$ is flexible in codimension one.
Equivalently, such a cone has an open subset with an infinitely transitive action of the special automorphism group on it. 
\end{abstract}

\maketitle

\input{1-intro}

\input{2-flex-criterion}

\input{3-cubics}

\input{4-cones}

\input{5-cylinders}

\input{6-collections}

\input{7-flexibility}

\bibliographystyle{abbrv}
\bibliography{references} 
\end{document}

%% file: common/packages.tex
\usepackage{amsmath}
\usepackage{amssymb,amsthm,amscd}

\usepackage{fullpage}
\usepackage[lmargin=3cm,rmargin=3cm,tmargin=2cm,bmargin=2cm, marginpar=2cm ]{geometry}

\usepackage{xcolor}
\usepackage{hyperref}
\hypersetup{
    colorlinks=true,
    linkcolor=blue,
    citecolor=cyan!80!black,
    filecolor=cyan!80!black,      
    urlcolor=cyan!80!black,
} 

\usepackage{extarrows} 

\usepackage[normalem]{ulem} 
\usepackage[toc,page]{appendix}

\usepackage{graphicx}

\usepackage{tikz}
\usetikzlibrary{matrix,arrows,positioning,calc,chains}

\usepackage{enumerate} 

%% file: common/notation.tex
\def\Aut{\operatorname{Aut}}

\def\SAut{\operatorname{SAut}}

\def\Pic{\operatorname{Pic}}

\def\AffCone{\operatorname{AffCone}}
\def\Ample{\operatorname{Ample}}
\def\Eff{\operatorname{\overline{NE}}}

\def\PP{{\mathbb P}}
\def\O{{\mathcal O}}

\def\QQ{{\mathbb Q}}
\def\NN{{\mathbb N}}

\def\G{{\mathbb G}}

\def\AA{{\mathbb A}}
\def\PP{{\mathbb P}}
\def\V{{\mathbb V}}
\def\K{{\mathbb K}}
\def\U{{\mathcal U}}
\def\V{{\mathcal V}}

\def\0{\circ}

\def\Aut{\operatorname{Aut}}
\def\SAut{\operatorname{SAut}}

\def\Pic{\operatorname{Pic}}

\def\AffCone{\operatorname{AffCone}}
\def\Ample{\operatorname{Ample}}
\def\relint{\operatorname{rel.int.}}
\def\Cone{\operatorname{Cone}}
\def\Eff{\operatorname{\overline{NE}}}

\def\Pol{\operatorname{Pol}}
\def\Forb{\operatorname{Forb}}

%% file: common/environments.tex
\theoremstyle{plain}
\newtheorem{theorem}{Theorem}[section]
\newtheorem{lemma}[theorem]{Lemma}

\newtheorem*{conjecture*}{Conjecture}
\newtheorem{question}[theorem]{Question}
\newtheorem{proposition}[theorem]{Proposition}

\newtheorem{corollary}[theorem]{Corollary}
\theoremstyle{definition}
\newtheorem{definition}[theorem]{Definition}
\newtheorem{example}[theorem]{Example}
\newtheorem{remark}[theorem]{Remark}

\newtheorem{sit}[theorem]{}

%% file: common/collab.tex
\usepackage[draft,multiuser,english]{fixme}

\fxsetup{layout=inline}
\FXRegisterAuthor{p}{ap}{AP\!\!}

\FXRegisterAuthor{g}{vg}{VG\!\!}

%% file: common/contacts.tex
\author{Alexander Perepechko}

\address{
  Kharkevich Institute for Information Transmission Problems\\
  19 Bolshoy Karetny per., 127994 Moscow, Russia
}

\address{
Moscow Institute of Physics and Technology (State University)\\
9 Institutskiy per., Dolgoprudny, Moscow Region, 141701, Russia 
}

\address{
National Research University Higher School of Economics\\
 20 Myasnitskaya ulitsa, Moscow 101000, Russia 
}

\email{a@perep.ru}

%% file: 1-intro.tex
\section{Introduction}
We are working over an algebraically closed field $\K$ of characteristic zero. 
Let $X$ be an affine algebraic variety over $\K$.
A point $p\in X$ is called \emph{flexible} if the tangent space $T_pX$ is spanned by the tangent vectors to the orbits of actions of the additive group of the field $\G_a=\G_a(\K)$ on $X$. 
The variety $X$ is called \emph{flexible} if each smooth point of $X$ is flexible, see \cite{AFKKZ}. 
All the $\G_a$-actions on $X$ generate the \emph{special automorphism group} $\SAut X\subset\Aut X$.

 A group $G$ is said to act on a set $S$ \emph{infinitely transitively} if it acts transitively on the set of ordered $m$-tuples of pairwise distinct points in $S$ for any $m\in\NN$.

 The following theorem explains the significance of the flexibility concept.
\begin{theorem}[{\cite[Theorem 0.1]{AFKKZ}}]\label{th:AFKKZ}
  Let $X$ be an affine algebraic variety of dimension $\ge2$. Then the following conditions are equivalent:
  \vspace{-2pt}
  \begin{enumerate}
    \item The variety $X$ is flexible;
    \item the group $\SAut X$ acts transitively on the smooth locus $X_{\mathrm{reg}}$ of $X$;
    \item the group $\SAut X$ acts infinitely transitively on $X_{\mathrm{reg}}$.
  \end{enumerate}
\end{theorem}

 The flexibility of affine cones over del Pezzo surfaces of degree $\ge4$ corresponding to any very ample polarization was confirmed in \cite{AKZ}, \cite{PW16}, \cite{P13}. 
 In this paper we study flexibility of affine cones over smooth del Pezzo surfaces of degree $3$. 
 
 A plurianticanonical polarization is a polarization with respect to a divisor proportional to the anticanonical one.
 It is shown in \cite[Corollary~1.8]{CPW13} that the affine cone $X$ over any plurianticanonically polarized del Pezzo surface of degree $3$  admits no $\G_a$-action. 
Nevertheless, $X$ admits a $\G_a$-action for any other very ample polarization, see \cite{CPW15}.
We use constructions and results from \cite{CPW15} in order to show that $\SAut(X)$ acts on $X$ with an open orbit, see Theorem~\ref{th:flex}. 
Then $X$ contains an open subset of flexible points,
that is, $X$ is \emph{flexible in codimension one}, see Corollary~\ref{c:flex}.

In Section 2 we provide a criterion of flexibility in codimension one for affine cones, which generalizes similar flexibility criteria in our previous articles \cite{MPS} and \cite{P13}. 
In Section 3 we recall generalities on cubic surfaces. 
In Section 4 we introduce a useful subdivision of the effective cone. 
In Section 5 we recall two constructions of cylinders. 
In Section 6 for each of the considered smaller cones we introduce a collection of cylinders.
In order to find these collections, we used the computer algebra system SageMath \cite{Sage}, especially the module for rational polyhedral cones created by V.~Braun and A.~Novoseltsev.
In Section 7 we use these collections to prove flexibility in codimension one and provide an example of a flexible cone.

The author is grateful to M.G.~Zaidenberg for everlasting motivation, numerous discussions and remarks, to I.~Cheltsov and J.~Park for useful discussions on the subject, and to I.~Arzhantsev for useful remarks and suggestions.

%% file: 2-flex-criterion.tex
\section{Flexibility criterion for affine cones}
Here we generalise flexibility criteria for affine cones \cite[Theorem 1.4]{MPS} and \cite[Theorem 5]{P13}.
In this section we let $Y$ denote a normal projective variety, $H$ a very ample divisor on $Y$, and $X=\AffCone_H(Y)$ the affine cone over $Y$ corresponding to $H$. By $\pi\colon X\setminus\{0\}\to Y$ we denote the projectivizing map.

\begin{definition}[{\cite[Definition 3.1.7]{KPZ11}}]
An open subset $U\subset Y$ is called $H$-\emph{polar} if
the complement $Y\setminus U$ is the support of some effective $\QQ$-divisor equivalent to $H$ over $\QQ$.
\end{definition}

\begin{lemma}\label{lm:inv}
Let $U$ be an $H$-polar affine open subset of $Y$ and $S$ be an $\SAut(X)$-orbit on $X$. 
Then $\pi(S)\cap U$ is $\SAut(U)$-invariant.
\end{lemma}
\begin{proof}
Assume the contrary. 
Then there exists a $\G_a$-action $G$ on $U$ that does not preserve $\pi(S)\cap U$. 
By \cite[Lemma 1.3]{MPS}, there exists a corresponding $\G_a$-action $\hat G$ on $X$ trivial on $X\setminus \pi^{-1}(U)$ such that $\pi$ sends $\hat G$-orbits to $G$-orbits. 
Thus, $\hat G$ does not preserve $S\cap\pi^{-1}(U)$. 
So, $\hat G$ does not preserve $S$, a contradiction.
\end{proof}

\begin{definition}
Let $\mathcal{U}=\{U\}$ be a collection of open affine subsets of $Y$.
\begin{enumerate}
    \item We say that a subset $Z$ of $Y$ is $\mathcal{U}$-\emph{invariant} if for any $U\in\mathcal{U}$ the intersection $U\cap Z$ is $\SAut(U)$-invariant. 
    \item We say that $\mathcal{U}$ is \emph{transversal} if $\bigcup_{U\in \mathcal{U}} U$ does not admit nontrivial $\mathcal{U}$-invariant subsets.
    \item We say that $\mathcal{U}$ is $H$-\emph{complete} if for any (not necessarily effective) $\QQ$-divisor equivalent to $H$ its support is not contained in the complement $Y\setminus\bigcup_{U\in\mathcal{U}} U$.
\end{enumerate}
\end{definition}

\begin{theorem}\label{th:flex}
 Let $\mathcal{U}$ be a transversal collection of $H$-polar open affine subsets of a normal projective variety $Y$ for some very ample divisor $H$. 
 Then there exists an $\SAut(X)$-orbit $S$ on the corresponding affine cone $X$ whose image contains $\bigcup_{U\in\mathcal{U}}U$.
  If, moreover, $\mathcal{U}$ is $H$-complete, then $S$ is open in $X$ and contains the open subset $\pi^{-1}(\bigcup_{U\in\mathcal{U}}U)\subset X$.
\end{theorem}
 \begin{proof}
Let $S\subset X$ be an $\SAut(X)$-orbit of some point $x\in\pi^{-1}(\bigcup_{U\in\mathcal{U}}U)$. By \cite[Proposition 1.3]{AFKKZ}, $S$ is locally closed. 
By Lemma~\ref{lm:inv}, the image $\pi(S)$ is $\mathcal{U}$-invariant. 
Since $\mathcal{U}$ is transversal and $\pi(S)\cap \bigcup_{U\in\mathcal{U}}U\ni \pi(x)$ is nonempty,
one has $\pi(S)\supset\bigcup_{U\in\mathcal{U}}U$.

Consider the natural $\G_m$-action on $X$ by homotheties. 
It sends $\SAut(X)$-orbits to $\SAut(X)$-orbits. 
The subgroup $\Gamma\subset\G_m$ that sends $S$ to itself either is finite, or equals $\G_m$. 
In the latter case $S=\pi^{-1}\circ\pi(S)\supset\pi^{-1}(\bigcup_{U\in\mathcal{U}}U)\subset X.$
Assume the former case, then $S$ is not open in $X$. 

Denote by $X^\times$ the blowup of $X$ at the origin. 
Then $X^\times$ is the total space of the line bundle $\O_Y(-H)$ over $Y$. Denote $\gamma=|\Gamma|$. 
The quotient $X'=X^\times/\Gamma$ is the total space of $\O_Y(-\gamma H)$ given by the map $t\mapsto t^{\gamma}$ in a local coordinate $t$ on each trivializing chart. 

Thus, we have the sequence of maps
\[X\dashrightarrow X^\times\to X'\to Y.\]
Let $S'$ be the Zariski closure of the image of $S$ in $X'$. Since $\pi(S)$ is open in $Y$, $S'$ intersects a generic fiber of $X'\to Y$ in a single point. Then $S'$ is a rational section of the line bundle $\O_Y(-\gamma H)$. The corresponding Cartier divisor $D$ is equivalent to $\gamma H$. 

Therefore, $\frac{1}{\gamma}D$ is a $\QQ$-divisor equivalent to $H$ such that its support is disjoint with $\pi(S)$.
So, $\mathcal{U}$ is not $H$-complete. 
This argument is essentially a generalisation of \cite[Lemma 1.2]{MPS}.
\end{proof}

%% file: 3-cubics.tex
\section{Cubic surfaces}

In the rest of this paper we denote by $Y$
a del Pezzo surface of degree 3. It is obtained by blowing up the projective plane $\PP^2$ at
six points $p_1,\ldots,p_6$ that do not belong to the same conic and no three of which lie on the same line, see \cite[Theorem IV.2.5]{Manin}. 

We denote by
\begin{itemize}
    \item  $\sigma$ the contraction $Y\to\PP^2$;
    \item  $E_1,\ldots,E_6$ the (-1)-curves contracted into $p_1,\ldots,p_6$ respectively;
    \item  $L_{i,j}$ a line passing through $p_i, p_j$ for given $i,j\in\{1,\ldots,6\}$;
    \item  $Q_i$ a conic passing through all the blown up points except $p_i$ for a given $i\in\{1,\ldots,6\}$;
    \item  $L$ the class of the proper transform of a general line on $\PP^2$;
    \item  $-K$ the anticanonical divisor class on $Y$;
    \item  $\Eff(Y)$ the Mori cone of numerically effective divisors; it is generated by all the (-1)-curves in $Y$;
    \item  $\Ample(Y)$ the ample cone; it contains all the ample divisors in its interior and is dual to $\Eff(Y)$.
\end{itemize}
We also denote the (-1)-curves in $Y$, which are proper transforms of $L_{i,j}$ and $Q_i$, by the same letters.
Given a $\QQ$-divisor $D$, we denote its divisor class by the same letter.
Note that 
\begin{itemize}
    \item $\Ample(Y)\subset\Eff(Y)$,
    \item $-K \equiv 3L-\sum_{i=1}^6E_i$,
    \item $L_{i,j}\equiv L-E_i-E_j$,
    \item $Q_j\equiv 2L-\sum_{i=1}^6E_i+E_j\equiv -K-L+E_j.$
\end{itemize}

%% file: 4-cones.tex
\section{Cones}
The 27 lines on a cubic surface were described by Cayley and Salmon in 1849, and later by Schl\"afli and Steiner in 1854, e.g. see \cite{Dolg27}.
The incidence graph of the (-1)-curves in $Y$ is the complement of the Schl\"afli graph, which is known to be 4-transitive, e.g. see  \cite{Cam}. That is, any isomorphism between two induced subgraphs on 4 vertices can be extended to the automorphism of the whole graph.
Thus,

\begin{remark}\label{r:4-hom}
Any ordered 4-tuple of disjoint (-1)-curves on $Y$ can be extended to an ordered 6-tuple of disjoint (-1)-curves, in the same manner as $(E_1,E_2,E_3,E_4)$ can be extended to $(E_1,\ldots,E_6)$. The latter fails for 5-tuples: the tuple $(E_1,\ldots,E_4,L_{5,6})$ is maximal. 
\end{remark}

\begin{definition}
We denote by $\relint(A)$ the relative interior of a rational polyhedral cone $A$.
\end{definition}

Given divisor classes $D_1,\ldots,D_n\in\Pic_\QQ(Y)$, we denote the cone generated by them by \[\Cone(D_1,\ldots,D_n)=\sum_{i=1}^n\QQ_{\ge0}D_i.\]

\begin{proposition}\label{pr:cones}
Let $H$ be a $\QQ$-divisor class in the interior of the Mori cone of $Y$. Then, up to an ordered choice of six disjoint (-1)-curves $E_1,\ldots,E_6$, it belongs to the relative interior of one of the following cones:
\begin{itemize}
    \item $B_6=\Cone(-K, E_1,E_2,E_3,E_4,E_5,E_6)$;
    \item $B_5=\Cone(-K, E_1,E_2,E_3,E_4,E_5)$;
    \item $B_5^\prime=\Cone(-K, E_1,E_2,E_3,E_4,L-E_5-E_6)$;
    \item $B_4=\Cone(-K, E_1,E_2,E_3,E_4)$;
    \item $B_3=\Cone(-K, E_1,E_2,E_3)$;
    \item $B_2=\Cone(-K, E_1,E_2)$;
    \item $B_1=\Cone(-K, E_1)$;
    \item $B_0=\Cone(-K)$;
    \item $C_5=\Cone(-K, L-E_6, E_1,E_2,E_3,E_4,E_5)$;
    \item $C_5^\prime=\Cone(-K, L-E_6, E_1,E_2,E_3,E_4,L-E_5-E_6)$;
    \item $C_4=\Cone(-K, L-E_6, E_1,E_2,E_3,E_4)$;
    \item $C_3=\Cone(-K, L-E_6, E_1,E_2,E_3)$;
    \item $C_2=\Cone(-K, L-E_6, E_1,E_2)$;
    \item $C_1=\Cone(-K, L-E_6, E_1)$;
    \item $C_0=\Cone(-K, L-E_6)$.
\end{itemize}
\end{proposition}
\begin{proof}
Consider a stellar subdivision of the Mori cone $\Eff(Y)$ along the anticanonical ray spanned by $-K$. 
Each cone of the subdivision fan is either a face of $\Eff(Y)$ or is generated by a face of $\Eff(Y)$ and the anticanonical ray. The former and the latter ones are in one-to-one correspondence: the corresponding cones coincide under taking quotient of the ambient space by $\langle K\rangle$.

Thus, a cone $C$ that contains $H$ in its relative interior is generated by the corresponding face $F$ of $\Eff(Y)$ and the anticanonical ray.
Following \cite[Section 2.1]{CPW15}, $F$ is the Fujita face of $H$ and is generated either by $r$ disjoint (-1)-curves (type $B(r)$ in terms of \cite{CPW15}), $r\le6$, or by 5 disjoint pairs of intersecting (-1)-curves (type $C(5)$ in terms of \cite{CPW15}).

Assume that $F$ is of type $B(r)$ and generated by disjoint (-1)-curves $D_1,\ldots,D_r$. 
By Remark~\ref{r:4-hom}, we may assume that $D_i=E_i$ for all $i\le\min(4,r)$. If $r\le4$, then the cone $C$ equals $B_r$. 

For $r>4$, there are only three  (-1)-curves disjoint with $E_1,\ldots,E_4$, namely, $E_5,E_6,L_{5,6}$. 
Since one may swap $E_5$ and $E_6$ by a choice of order, $C$ equals $B_5$ or $B_5^\prime$ for $r=5$. 
Finally, $C$ equals $B_6$ for $r=6$, because $L_{5,6}$ intersects both $E_5$ and $E_6$.

Assume now that $F$ is of type $C(5)$ and is generated by 5 disjoint pairs of intersecting (-1)-curves $D_i$ and $D_i^\prime$, $i=1,\ldots,5$. 
That is, $D_i\cdot D_j=D_i^\prime\cdot  D_j^\prime=D_i\cdot D_j^\prime=0$ for $i\neq j$ and $D_i\cdot D_i^\prime=1$. 

Let 
\[H\equiv k(-K)+\sum_{i=1}^5(d_iD_i+d_i^\prime D_i^\prime),\]
where $k,d_1,d_1^\prime,\ldots,d_5,d_5^\prime\in\QQ_{>0}$.
Without loss of generality, $d_i-d_i^\prime$ is positive for $i\le r$ and zero for $i>r$, where $r\in\{0,1,\ldots,5\}$. 
 
By Remark~\ref{r:4-hom}, we may assume that $D_i=E_i$ for $i\le4$. 
The only (-1)-curves disjoint with $E_1,\ldots,E_4$ are $E_5,E_6,L_{5,6}$.
Thus, the (unordered) pair $\{D_5, D_5^\prime\}$ coincides either with $\{E_5,L_{5,6}\}$ or with $\{E_6,L_{5,6}\}$. Reordering $E_5,E_6$ if necessary, we may assume the former. We infer that $D_i^\prime=L_{5,6}$ for $i\le4$, because it is a unique (-1)-curve that meets $E_i$ and is disjoint with each $E_j$, where $j\neq i, j\le5$. In particular, $D_i+D_i^\prime\equiv L-E_6$ for all $i=1,\ldots,5.$ Thus,
\[H\equiv k(-K)+\left(\sum_{i=1}^5 d_i^\prime\right)(L-E_6)+ \sum_{i=1}^r(d_i-d_i^\prime)D_i,\]
where all coefficients are positive, cf. \cite[(2.1.4)]{CPW15}.
In the case $r\le 4$, $H\in\relint(C_r)$. For $r=5$, either $D_5=E_5$ and $H\in\relint(C_r)$, or $D_5=L_{5,6}$ and $H\in\relint(C_r^\prime)$.
In fact, we have a stellar subdivision of the cone $C$ along $L-E_6$.
 \end{proof}

%% file: 5-cylinders.tex
\section{Cylinders}

\begin{definition}[{\cite[Definition 3.1.5]{KPZ11}}]
An open subset $U$ of a normal projective variety is called a \emph{cylinder} if
$U\cong \AA^1\times Z$ for some smooth variety $Z$ with $\Pic(Z)=0$. 
By \emph{fibers} of $U$ we mean $\AA^1$-fibers of the projection to $Z$. 
There is a natural $\G_a$-action on $U$, whose orbits are fibers of $U$.
\end{definition}

We introduce below two known constructions of cylinders on a cubic surface $Y$.

\begin{sit}
Let $i$ be equal to either 1 or 6.
Consider two tangent lines $T_i,T_i^\prime$ from $p_i$ to $Q_i$ on $\PP^2$. 
The pencil generated by $2T_i$ and $Q_i$ consists of conics tangent to $T_i$ at $p=T_i\cap Q_i$. 
Its only base point is $p$. 

Thus, the complement 
\[U_i=Y\setminus\sigma^{-1}(Q_i\cup T_i)\cong\AA^1\times\AA^1_*\] 
is a cylinder. Similarly, we introduce another cylinder 
\[U_i^\prime=Y\setminus\sigma^{-1}(Q_i\cup T_i^\prime).\]
Note that the complement to $U_i\cup U_i^\prime$ in $Y$ consists of (-1)-curves $E_1,\ldots,E_6,Q_i$.
\end{sit}

\begin{sit}\label{sit:P1-cyl}
Consider the following cylinder as in \cite[Example 4.1.6]{CPW15} and \cite[Lemma 4.2.3, Case 2]{CPW15}. 
Let $i,j\in\{1,\ldots,6\}$, $i\neq j$. 
Let $\sigma'\colon Y\to\PP^1\times\PP^1$ be the morphism contracting curves $E_k$, $k\in\{1,\ldots,6\}\setminus\{i,j\},$ and $L_{i,j}$ into points $p_1,p_2,p_3,p_4,p_{i,j}$ respectively. 
Choose coordinates on $\PP^1\times\PP^1$ so that images of $E_i$ and $E_j$ are curves of bidegrees $(1,0)$ and $(0,1)$ respectively.

There exists an irreducible curve $Q$ of bidegree $(2,1)$ that passes through $p_1,p_2,p_3,p_4$, $p_{i,j}$. 
One may check that $Q$ is the image of $Q_i$. 
Indeed, the proper transforms in $Y$ of general curves of bidegrees $(1,0)$ and $(0,1)$ are equivalent to $E_i+L_{i,j}\equiv L-E_j$ and $E_j+L_{i,j}\equiv L-E_i$ respectively. 
Thus, the proper transform $\tilde Q$ of $Q$ in $Y$ satisfies
\begin{align*}
 \tilde Q\cdot E_k&=1, \qquad k\neq i,j,   \\
 \tilde Q\cdot (L-E_j)&=1,\\
 \tilde Q\cdot(L-E_i)&=2,\\
 \tilde Q\cdot \tilde Q = Q\cdot Q-5&=-1.
\end{align*}
 There is a unique (-1)-curve on $Y$ satisfying these relations, namely, $Q_i$.

There are two curves $T,T'$ of bidegree $(0,1)$ tangent to $Q$.
Let $q=T\cap Q$ and let $R$ be a curve of bidegree $(1,0)$ passing through $q$. 
Consider a pencil generated by $Q$ and by $2R+T$. 
Its general fiber is an irreducible rational curve of bidegree $(2,1)$ that has contact with $Q$ at $q$ of order 2, so $q$ is the only base point of this pencil.
Then we have a cylinder $\PP^1\times\PP^1\setminus(T\cup Q\cup R).$ We denote its isomorphic preimage in $Y$ by $V_{i,j}$. 

The same construction for $T'$ yields another cylinder $V_{i,j}^\prime\subset Y$.
Note that the complement to $V_{i,j}\cup V_{i,j}^\prime$ in $Y$ consists of contracted (-1)-curves $E_k$, $k\in\{1,\ldots,6\}\setminus\{i,j\},$ $L_{i,j}$, the curve $Q_i$, and the preimages of points $T'\cap R,\,T\cap R'$, where $R'$ is the curve of bidegree $(1,0)$ passing through $T'\cap Q$.
\end{sit}

%% file: 6-collections.tex
\section{Collections}
\begin{definition}
Given an open subset $U\subset Y$, which complement consists of components $D_1,\ldots,D_n$, we define the \emph{polarity cone}
$\Pol(U)=\Cone(D_1,\ldots,D_n)$. Then $U$ is $H$-polar for any divisor $H$ from $\relint(\Pol(U))$.

Given a collection $\U$ of open subsets of $Y$, we define the \emph{polarity cone} via \[\Pol(\U)=\bigcap_{U\in\U}\Pol(U).\] Let $D_1,\ldots,D_n$ be the irreducible curves in $Y\setminus \bigcup_{U\in\U}U$. We define the \emph{forbidden cone} via 
\[\Forb(\U)=\Cone(D_1,\ldots,D_n).\]
Then $\U$ is $H$-complete for any $H\notin\Forb(\U)$.
\end{definition}

\begin{sit}
Let us introduce the collections
\begin{align*}
    \U_1 &= \{U_1,U_1^\prime\},\\
    \U_2 &= \{U_6,U_6^\prime\},\\
    \U_3 &= \{V_{5,6},V_{5,6}^\prime\},\\
    \U_4 &= \{V_{5,6},V_{5,6}^\prime,V_{6,5},V_{6,5}^\prime\},\\
    \U_5 &= \{V_{5,6},V_{5,6}^\prime,V_{4,6},V_{4,6}^\prime\}.
\end{align*}
If the tangent line in our construction of one or more cylinders passes through a contracted point, then the polarity cone of such cylinder might become only bigger, and the forbidden cones of considered collections do not change. Thus, we assume that no tangent line passes through a contracted point.
Then their polarity cones are
\begin{align*}
    \Pol(\U_1) =& \Cone(E_1,\ldots,E_6,Q_1,L-E_1),\\
    \Pol(\U_2) =& \Cone(E_1,\ldots,E_6,Q_6,L-E_6),\\
    \Pol(\U_3) =& \Cone(E_1,E_2,E_3,E_4,L_{5,6},Q_5,L-E_5,L-E_6),\\
    \Pol(\U_4) =& \Cone(E_1,E_2,E_3,E_4,L_{5,6},Q_5,L-E_5,L-E_6)\\
    &\cap\Cone(E_1,E_2,E_3,E_4,L_{5,6},Q_6,L-E_5,L-E_6),\\
    \Pol(\U_5) =& \Cone(E_1,E_2,E_3,E_4,L_{5,6},Q_5,L-E_5,L-E_6)\\
    &\cap\Cone(E_1,E_2,E_3,E_5,L_{4,6},Q_4,L-E_4,L-E_6).
\end{align*}
All these polarity cones are full-dimensional. Thus, $\U_i$ consists of $H$-polar divisors for any $H\in\relint(\Pol(\U_i))$, where $i\in\{1,\ldots,5\}$.
The forbidden cones are
\begin{align*}
    \Forb(\U_1) =& \Cone(E_1,\ldots,E_6,Q_1),\\
    \Forb(\U_2) =& \Cone(E_1,\ldots,E_6,Q_6),\\
    \Forb(\U_3) =& \Cone(E_1,E_2,E_3,E_4,L_{5,6},Q_5),\\
    \Forb(\U_4) =& \Cone(E_1,E_2,E_3,E_4,L_{5,6}),\\
    \Forb(\U_5) =& \Cone(E_1,E_2,E_3).
\end{align*}
\end{sit}

\begin{proposition}\label{pr:tr}
Collections $\U_1,\ldots,\U_5$ are transversal.
\end{proposition}
\begin{proof}
Consider a pair of cylinders $U_i,U_i^\prime$ for any $i$.
Any fiber of $U_i$ meets all fibers of $U_i^\prime$ and vice versa. Thus, $\U_1$ and $\U_2$ are transversal.

The same holds for a pair $V_{i,j},V_{i,j}^\prime$ for any $i,j$.
Thus, $\U_3$ is transversal. 

Finally, a union of transversal collections of open subsets is transversal. Thus, $\U_4$ are $\U_5$ are also transversal.
\end{proof}

\begin{lemma}\label{l:subcone}
Let $A,B$ be two rational polyhedral cones in $\QQ^n$ for some $n\in\NN$. Assume that $A\subset B$ and there exists $a\in A$ that belongs to $\relint(B)$. Then $\relint(A)\subset\relint(B)$.
\end{lemma}
\begin{proof}
Denote the generating vectors of $A$ by $a_1,\ldots,a_r$ and of $B$ by $b_1,\ldots,b_s$. Consider an arbitrary $x\in\relint(A)$. Then 
\begin{enumerate}
    \item $a_i$ can be expressed as a linear combination of $b_1,\ldots,b_s$ with non-negative coefficients,
    \item $a$ can be expressed as a linear combination of $b_1,\ldots,b_s$ with positive coefficients, and
    \item  $x$ can be expressed as a linear combination of $a_1,\ldots,a_r,a$ with positive coefficients.
\end{enumerate}
Substituting such expressions for $a_1,\ldots,a_r,a$ into an expression for $x$, we obtain a linear combination of $b_1,\ldots,b_s$ with positive coefficients that equals $x$. Thus, $x\in\relint(B)$.
\end{proof}

\begin{proposition}\label{pr:U1}
Consider an ample $\QQ$-divisor $H$ from $\relint(C)$, where $C$ is one of the cones $B_1,\ldots,B_6$.
 Then $\U_1$ is an $H$-complete collection of $H$-polar subsets.
\end{proposition}
\begin{proof}
Since $-K\equiv Q_1+(L-E_1),$ all generators of $C$ belong to $\Pol(\U_1)$. Since 
\[
  \relint(\Pol(\U_1))\ni\sum_{i=1}^6E_i+4Q_1+(L-E_1)
  \equiv 9L-3\sum_{i=2}^6E_i
  \equiv 3(-K+E_1)\in C,
\]
$\relint(\Pol(\U_1))\supset\relint(C)$ by Lemma~\ref{l:subcone}.
Thus, $\U_1$ consists of $H$-polar cylinders.

The cone $\Forb(\U_1)$ contains no ample divisor classes, since $D\cdot E_1\le0$ for any $D\in\Forb(\U_1)$. Thus, $\U_1$ is $H$-complete.
\end{proof}

\begin{proposition}\label{pr:U2}
Consider an ample $\QQ$-divisor $H$ from $\relint(C_5)$.
 Then $\U_2$ is an $H$-complete collection of $H$-polar subsets.
\end{proposition}
\begin{proof}
Since $-K\equiv Q_6+(L-E_6),$ all generators of $C$ belong to $\Pol(\U_2)$. Since 
\[
  \relint(\Pol(\U_2))\ni\sum_{i=1}^5E_i+2E_6+Q_6+4(L-E_6)
  \equiv 6L-2E_6
  \equiv 2(-K+\sum_{i=1}^5E_i)\in C_5,
\]
$\relint(\Pol(\U_2))\supset\relint(C_5)$ by Lemma~\ref{l:subcone}.
Thus, $\U_2$ consists of $H$-polar cylinders.

The cone $\Forb(\U_2)$ contains no ample divisor classes, since $D\cdot E_6\le0$ for any $D\in\Forb(\U_2)$. Thus, $\U_2$ is $H$-complete.
\end{proof}

\begin{sit}
Denote \[C_4^\prime=\{H\in C_4\mid H\cdot(L+E_5-2E_6)=0\}.\]
It is the cone generated by $E_1,E_2,E_3,E_4,$ and $-K+2(L-E_6)$.
\end{sit}

\begin{proposition}\label{pr:U3}
Consider an ample $\QQ$-divisor $H$ from $\relint(C4)\setminus\relint(C4')$.
 Then $\U_3$ is an $H$-complete collection of $H$-polar subsets.
\end{proposition}
\begin{proof}
Since $-K\equiv Q_5+(L-E_5),$ all generators of $C$ belong to $\Pol(\U_3)$. Since 
\begin{multline*}
  \relint(\Pol(\U_3))\ni\sum_{i=1}^4E_i+L_{5,6}+3Q_5+(L-E_5)+(L-E_6)\\
  \equiv 9L-2\sum_{i=1}^5E_i-5E_6
  \equiv 2(-K)+3(L-E_6)\in C_4,
\end{multline*}
$\relint(\Pol(\U_3))\supset\relint(C_4)$ by Lemma~\ref{l:subcone}.
Thus, $\U_3$ consists of $H$-polar cylinders.

If $H\in\Forb(\U_3)$, then $H\cdot(L+E_5-2E_6)=0$ and $H\in C_4^\prime$, which contradicts our assumption. Thus, $\U_3$ is $H$-complete.
\end{proof}

\begin{proposition}\label{pr:U4}
Consider an ample $\QQ$-divisor $H$ from $\relint(C)$, where $C$ equals either $B_5^\prime$ or $C_5^\prime$.
 Then $\U_4$ is an $H$-complete collection of $H$-polar subsets.
\end{proposition}
\begin{proof}
Since $-K\equiv Q_5+(L-E_5)\equiv Q_6+(L-E_6),$ all generators of $C$ belong to $\Pol(\U_4)$. Since 
\begin{multline*}
  \relint(\Pol(\U_4))\ni
  \sum_{i=1}^4E_i+L_{5,6}+Q_5+3(L-E_5)+2(L-E_6)\\
  \equiv \sum_{i=1}^4E_i+L_{5,6}+Q_6+2(L-E_5)+3(L-E_6)\\
  \equiv 8L-4E_5-4E_6
  \equiv 2(-K+\sum_{i=1}^4E_i+L_{5,6})\in C,
\end{multline*}
$\relint(\Pol(\U_4))\supset\relint(C)$ by Lemma~\ref{l:subcone}.
Thus, $\U_4$ consists of $H$-polar cylinders.

The cone $\Forb(\U_4)$ contains no ample divisor classes, since $D\cdot E_1\le0$ for any $D\in\Forb(\U_4)$. Thus, $\U_4$ is $H$-complete.
\end{proof}

\begin{proposition}\label{pr:U5}
Consider an ample $\QQ$-divisor $H$ from $\relint(C)$, where $C$ equals one of the cones $C_4^\prime,C_3,C_2,C_1,C_0$.
 Then $\U_5$ is an $H$-complete collection of $H$-polar subsets.
\end{proposition}
\begin{proof}
Since $-K\equiv Q_5+(L-E_5)\equiv Q_4+(L-E_4),$ all generators of $C$ belong to $\Pol(\U_5)$. Since 
\begin{multline*}
  \relint(\Pol(\U_5))\ni
  \sum_{i=1}^4E_i+L_{5,6}+3Q_5+(L-E_5)+2(L-E_6)\\
  \equiv \sum_{i=1}^3E_i+E_5+L_{4,6}+3Q_4+(L-E_4)+2(L-E_6)\\
  \equiv 10L-2\sum_{i=1}^5E_i-6E_6
  \equiv 2(-K+2(L-E_6))\in C,
\end{multline*}
$\relint(\Pol(\U_5))\supset\relint(C)$ by Lemma~\ref{l:subcone}.
Thus, $\U_5$ consists of $H$-polar cylinders.

The cone $\Forb(\U_5)$ contains no ample divisor classes, since $D\cdot E_6\le0$ for any $D\in\Forb(\U_5)$. Thus, $\U_5$ is $H$-complete.
\end{proof}

%% file: 7-flexibility.tex
\section{Flexibility}

\begin{theorem}
Let $Y$ be a smooth del Pezzo surface of degree 3, and $H$ be a very ample divisor on $Y$ which is not linearly equivalent to $-dK_Y$ for any $d\in\QQ$. Then the  group  $\SAut(X)$ of the affine cone $X=\AffCone_HY$ acts on $X$ with an open orbit.
\end{theorem}
\begin{proof}
By Proposition~\ref{pr:cones}, one may assume that $H$ belongs to the relative interior of one of cones $B_1,\ldots,B_6,B_5^\prime,C_0,\ldots,C_5,C_5^\prime.$ By Propositions~\ref{pr:U1}, \ref{pr:U2}, \ref{pr:U3}, \ref{pr:U4}, and \ref{pr:U5}, $Y$ admits an $H$-complete collection of $H$-polar cylinders, which is also transversal by Proposition~\ref{pr:tr}. By Theorem~\ref{th:flex}, $X$ contains an open $\SAut(X)$-orbit.
\end{proof}
\begin{corollary}\label{c:flex}
The affine cone $X$ over a smooth del Pezzo surface of degree 3 polarized by a very ample non-plurianticanonical divisor is flexible in codimension one.
\end{corollary}
\begin{proof}
The statement follows from \cite[Corollary 1.10 and Theorem 2.2]{AFKKZ}.
\end{proof}

Though $\SAut(X)$ acts on $X$ with open orbit, the complement to this orbit is unknown in general. Thus,

\begin{question}
Which polarized cubic surfaces provide a flexible cone? And which ones provide a cone flexible in codimension two?
\end{question}
Below we provide an example of a flexible cone.

\begin{sit}\label{sit:P1-cyl-D}
Let $(D_1,\ldots,D_5)$ be a maximal 5-tuple of (-1)-curves and $Q$ be one of two (-1)-curves intersecting all $D_1,\ldots,D_5$.
Consider the cylinder construction \ref{sit:P1-cyl} for the contraction $\sigma'\colon Y\to\PP^1\times\PP^1$ of curves $D_1,\ldots,D_5$ and the image of $Q$ as the curve of bidegree $(2,1)$.
We denote by $\V(D_1,\ldots,D_5\mid Q)$ the pair of cylinders given by such construction.
For example, $\U_3 = \V(E_1,E_2,E_3,E_4,L_{5,6}\mid Q_5)$.

Assume that tangent lines $T,T'$ are images of (-1)-curves. 
Equivalently, points $q=Q\cap T$ and $q'=Q\cap T'$ are images of (-1)-curves, say, $D_i$ and $D_j$ respectively. 
Thus, (-1)-curves $Q,T,D_i$ intersect in one point in $Y$, and same for $Q,T',D_j$.
So, the tangent lines $T,T'$ are images of some (-1)-curves if and only if  $Q\cap T\cap D_i$ and $Q\cap T'\cap D_j$ are triple intersections (called \emph{Eckardt points}, see \cite[Section 9.1.4]{Dolg-CAG}) in $Y$ for some $i,j\in\{1,\ldots,5\}$.

Then $R,R'$ are also images (-1)-curves.
As in \ref{sit:P1-cyl}, the complement to $\V(D_1,\ldots,D_5\mid Q)$ consists of (-1)-curves $D_1,\ldots,D_5,Q$ and two intersections of (-1)-curves $T\cap R'$ and $T'\cap R$.
\end{sit}

\begin{example}
Let $Y$ be the Fermat surface $z_0^3+z_1^3+z_2^3+z_3^3=0$ in $\PP^3$. Let us take exceptional curves as follows, e.g. see \cite[Example 4.7]{BK}:
\begin{align*}
    E_1=&\{wz_0+z_1=wz_2+z_3=0\},\\
    E_2=&\{z_0+z_1=z_2+wz_3=0\},\\
    E_3=&\{z_0+wz_1=z_2+z_3=0\},\\
    E_4=&\{z_0+z_3=z_1+z_2=0\},\\
    E_5=&\{z_0+wz_3=z_1+wz_2=0\},\\
    E_6=&\{wz_0+z_3=wz_1+z_2=0\},
\end{align*}
where $w$ is a primitive cubic root of unity.
The 18 Eckardt points on $Y$ are the following:
\begin{align*}
    E_i\cap Q_j\cap L_{i,j}, &\quad \mbox{for all } i,j\in\{1,2,3\}, i\neq j; \\
    E_i\cap Q_j\cap L_{i,j}, &\quad \mbox{for all } i,j\in\{4,5,6\}, i\neq j; \\
    L_{1,i}\cap L_{2,j}\cap L_{3,k}, &\quad \mbox{for all $i,j,k$ such that } \{i,j,k\}=\{4,5,6\}.
\end{align*}

Let us consider 4 pairs of cylinders $\V(D_1,\ldots,D_5\mid Q)$ as in \ref{sit:P1-cyl-D} corresponding to the following data, identifying $T,T',R,R'$ with their proper transforms in $Y$.
\begin{center}
\begin{tabular}{ c | c | c c c | c  c  c}
 $D_1,\ldots, D_5$ & $Q$ & $\sigma^{-1}(q)$ & $T$ & $R$ & $\sigma^{-1}(q')$ & $T'$ & $R'$ \\
 \hline
 $E_1,E_2,E_3,E_4,L_{5,6}$ & $Q_6$ 
 & $E_4$ & $L_{4,6}$ & $L_{4,5}$
 & $L_{5,6}$ & $E_5$ & $E_6$ \\ 
 $L_{1,5},L_{2,5},L_{3,5},E_4,E_6$ & $L_{4,6}$ 
 & $E_4$ & $Q_6$ & $L_{4,5}$
 & $E_6$ & $Q_4$ & $L_{5,6}$ \\ 

$L_{2,3},L_{2,6},L_{3,6},E_1,Q_1$ & $L_{1,5}$ 
 & $L_{2,6}$ & $L_{3,4}$ & $L_{3,5}$
 & $L_{3,6}$ & $L_{2,4}$ & $L_{2,5}$ \\ 

$L_{1,3},L_{3,4},L_{3,6},E_2,E_5$ & $L_{2,5}$ 
 & $L_{3,4}$ & $L_{1,6}$ & $E_4$
 & $L_{3,6}$ & $L_{1,4}$ & $E_6$ 
\end{tabular}
\end{center}
We denote by $\V$ their union, which is a collection of 8 cylinders.
In order to check that $Y=\bigcup_{U\in\V} U,$ it is enough to verify that

In order to prove that $Y=\bigcup_{U\in\V} U$, it suffices (and is easy) to check directly that the complement $Y\setminus\bigcup_{U\in\V}U$ contains no
\begin{itemize}
    \item (-1)-curve;
    \item intersection point of a pair of (-1)-curves;
    \item Eckardt point.
\end{itemize}
Let us take
\[H\equiv 13L-E_1-4E_2-4E_3-E_4-8E_5-3E_6.\]
We claim that $H$ belongs to the interior of $\Pol(\V)$, i.e., to the interior of the polarity cone of each cylinder in $\V$. 
Then the affine cone $\AffCone_H(Y)$ is flexible.
\end{example}